\documentclass[12pt]{amsart}

\usepackage{amsmath} 
\usepackage{amsfonts}
\usepackage{amssymb}
\usepackage{amsthm}
\usepackage{latexsym}
\usepackage{color}
\usepackage{enumerate}
\usepackage{mathrsfs}
\usepackage{setspace}
\usepackage{tikz}

\usepackage{setspace}
\let\oldtocsection=\tocsection
\let\oldtocsubsection=\tocsubsection
\renewcommand{\tocsection}[2]{\hspace{-1mm}\bf\oldtocsection{ #1}{#2}}
\renewcommand{\tocsubsection}[2]{\hspace{6.9mm}\oldtocsubsection{#1}{#2}}

\newtheorem{thm}{Theorem}[section]

\newtheorem{prop}{Proposition}[section]

\newtheorem{remark}{Remark}[section]
\numberwithin{equation}{section}

\usepackage{times}

\def\R{\mathbb{R}}

\def\R{\mathbb{R}}

\def\S{\Sigma}
\def\({\left(}
\def\){\right)}

\def\={\stackrel{(n=2)}{=}}
\def\p{\partial}

\newcommand{\be}{\begin{equation}}
\newcommand{\ee}{\end{equation}}
\newcommand{\bee}{\begin{equation*}}
\newcommand{\eee}{\end{equation*}}

\newcommand{\m}{\mathfrak{m}}

\allowdisplaybreaks

\begin{document}
	
	\title[]{Fill-ins with scalar curvature lower bounds and applications to positive mass theorems}

\author[S. McCormick]{Stephen McCormick}
\address{Institutionen f\"or teknikvetenskap och matematik \\
	Lule{\aa} tekniska universitet \\
	971 87 Lule\aa \\
	Sweden} 
\email{stephen.mccormick@ltu.se}

\begin{abstract}
	Given a constant $C$ and a smooth closed $(n-1)$-dimensional Riemannian manifold $(\S, g)$ equipped with a positive function $H$, a natural question to ask is whether this manifold can be realised as the boundary of a smooth $n$-dimensional Riemannian manifold with scalar curvature bounded below by $C$. That is, does there exist a \emph{fill-in} of $(\S,g,H)$ with scalar curvature bounded below by $C$?
	
	We use variations of an argument due to Miao and the author \cite{MM17} to explicitly construct fill-ins with different scalar curvature lower bounds, where we permit the fill-in to contain another boundary component provided it is a minimal surface. Our main focus is to illustrate the applications of such fill-ins to geometric inequalities in the context of general relativity. By filling in a manifold beyond a boundary, one is able to obtain lower bounds on the mass in terms of the boundary geometry through positive mass theorems and Penrose inequalities. We consider fill-ins with both positive and negative scalar curvature lower bounds, which from the perspective of general relativity corresponds to the sign of the cosmological constant, as well as a fill-in suitable for the inclusion of electric charge.
\end{abstract}
	\maketitle

\vspace{5mm}

	\tableofcontents
	
	\newpage 
	
\section{Introduction}
	Given a closed $(n-1)$-dimensional Riemannian manifold $(\S,g)$ equipped with a smooth function $H$, a natural question to ask is whether or not the triple $(\S,g,H)$ admits a non-negative scalar curvature fill-in. That is, {\em can $(\S,g)$ be realised as the boundary of an $n$-dimensional manifold $(\Omega,\gamma)$ with non-negative scalar curvature and outward-pointing mean curvature $H$?}. Recently the problem was highlighted by Gromov (Problem A in \cite{Gromov}) and it has been the subject of several recent works (for example, \cite{Jeff11,SWWZ,SWW}). Here we are predominantly interested in applications of such fill-ins to mathematical general relativity, where the fill-in $\Omega$ is often permitted to have another boundary component, provided that it is a minimal surface. Indeed in many situations where the problem arises, such a minimal surface inner boundary can itself be filled in and therefore causes no problems.
	
	In what follows, we will refer to a triple $(\S,g,H)$ as \textit{Bartnik data} following the literature on Bartnik's quasi-local mass, which concerns the closely related problem of non-negative scalar curvature \textit{extensions} -- that is, ``filling out" to infinity -- as opposed to fill-ins (see, for example, the review articles \cite{CCSurvey} and \cite{MyBartnikReview}). Rather than focusing purely on ``fill-ins'' with non-negative scalar curvature, it is interesting to ask if $(\S,g,H)$ can be realised as the boundary of a Riemannian manifold with any prescribed scalar curvature lower bound.
	
	This article serves to highlight an elementary approach to constructing fill-ins used in \cite{MM17} and demonstrate some new results using the same technique, with an emphasis on the applications of such fill-ins. For example, these fill-ins can lead to a kind of positive mass theorem with boundary, which we illustrate in the case of asymptotically hyperbolic manifolds (Theorem \ref{thm-PMTAHB}) as well as an example for asymptotically flat manifolds equipped with an electric field (Theorem \ref{thm-chargedpmt}). Since the type of fill-ins constructed have a minimal surface inner boundary, which means they also give a lower bound on the inner Bartnik mass -- a notion of quasi-local mass formulated by Bray \cite{BCinnermass} in the spirit of the Bartnik mass, replacing the aforementioned extensions with fill-ins. We state some of the main results more precisely below. The reader is directed to Section \ref{S-background} for more background on the problems considered here.

\begin{thm}[Fill-ins with negative scalar curvature lower bound]\label{thm-neg}
	{\, }	
	
\noindent Let $(\S,g,H)$ be $(n-1)$-dimensional Bartnik data with $H>0$, and $C<0$ be a negative constant such that
\bee
	R_g-2H\Delta H^{-1}>0
\eee 
and
\be \label{eq-neglambdacond}
	\max\left\{\max_\S \frac{n-2}{n-1}\,\frac{H^2}{R_g-2H\Delta H^{-1}}, \max_\S \frac{H^2r_o^2}{(n-1)^2}  \right\}<1-\frac{C}{n(n-1)} r_o^2
\ee 
where $R_g$ is the scalar curvature of $g$ and $r_o$ is the area radius of $(\S,g)$. Then there exists a compact manifold $(\Omega,\gamma)$ with boundary $\p\Omega=\Sigma_o\cup \Sigma_H$ where $\S_o$ is isometric to $(\S,g)$ with outward-pointing mean curvature $H$, $\S_H$ is minimal, and the scalar curvature $R_\gamma$ of $\gamma$ is greater than $C$.
\end{thm}

\begin{remark}
	The condition \eqref{eq-neglambdacond} is a technical condition that arises due to the construction itself and can be thought of as a kind of ``pointwise Hawking mass positivity". What we mean by this is that in the case that $R_g$ and $H$ are constant, then the condition reduces to simply positivity of the a higher dimensional Hawking mass tailored to the appropriate cosmological constant. Note that while the Hawking mass is generally considered in dimension $3$ and there are several inequivalent formulations of the Hawking mass in higher dimensions, when $H$ is constant they all agree with the expected value for spheres in Schwarzschild--anti-de Sitter manifolds. Namely, the expression
	
	\begin{equation*}
		m_{H}(\S,g,H;C)=\frac{r_o}{2}\left(1-\frac{r_o^2}{(n-1)^2}\(H^2+\frac{C(n-1)}{n}\)\right).
	\end{equation*}

\end{remark}

From this we are able to prove the following positive mass theorem for asymptotically hyperbolic manifolds with boundary using the same method as was used to obtain a Penrose-like inequality in \cite{MM17}. Specifically, we obtain the following.
\begin{thm}[Positive mass theorem for asymptotically hyperbolic manifold with boundary]\label{thm-PMTAHB}
	Let $(M,\gamma)$ be an asymptotically hyperbolic $n$-manifold that is spin, with scalar curvature $R_\gamma\geq C$ for some $C<0$, and inner boundary $\S=\p M$. Assume that the Bartnik data $(\S,g,H)$ induced on the boundary satisfies
	\bee
	R_g-2H\Delta H^{-1}>0
	\eee
	and
	\bee 
\max\left\{\max_\S 	\frac{n-2}{n-1}	\,\frac{H^2}{R_g-2H\Delta H^{-1}}, \max_\S \frac{H^2r_o^2}{(n-1)^2}  \right\}<1-\frac{C}{n(n-1)} r_o^2.
	\eee 
	Then $(M,\gamma)$ has positive asymptotically hyperbolic mass, in the sense of Wang \cite{WangMass}.
\end{thm}
\begin{remark}
	The definition of mass in the asymptotically hyperbolic case is somewhat subtle, however the precise definition will not be required here. This theorem follows directly from applying a known positive mass theorem to a manifold with corners, and for the sake of simplicity we apply the work of Bonini and Qing directly \cite{BoniniQing}, who prove a positive mass theorem for such manifolds with corners using Wang's definition of mass \cite{WangMass}.

	It is important to remark that the known positive mass theorem for asymptotically hyperbolic manifolds in fact already applies for manifolds with boundaries like those considered here provided that $H\leq \sqrt{\frac{-C(n-1)}{n}}$. So from this perspective, the result provides only a minor extension. The positive mass theorem with boundary presented here is more interesting when thought of as leading to a Penrose-like inequality assuming that the Riemannian Penrose inequality is established for asymptotically hyperbolic manifolds. Details of this Penrose-like inequality are given in Section \ref{S-apps}, however in order to rigorously prove it we would first require a proof of the Riemannian Penrose inequality in the asymptotically hyperbolic case, which remains an open problem.
\end{remark}

The next result demonstrates the existence of fill-ins with a non-negative scalar curvature lower bound.

\begin{thm}[Fill-ins with non-negative scalar curvature lower bound]\label{thm-pos}
	{\, }	
	
	Let $(\S,g)$ be a closed $(n-1)$-dimensional manifold, $H$ be a positive function. Suppose for some constant $C\geq 0$ we have
	\be \label{eq-poscond} 
R_g-2H\Delta H^{-1}>	 \frac{n(n-2)H^2}{n(n-1)-C r_o^2}
\ee	
	where $R_g$ is the scalar curvature of $g$ and $r_o$ is the area radius of $(\S,g)$, which we further ask satisfies $r_o<\sqrt{\frac{n(n-1)}{C}}$ when $C>0$. Then there exists a compact manifold with scalar curvature bounded below by $C$, whose boundary consists of two disconnected components, one being a minimal surface and the other isometric to $(\S,g)$ with outward-pointing mean curvature $H$.
\end{thm}
\begin{remark}
	Note that for coordinate $(n-1)$-spheres in an $n$-sphere, one would have equality in \eqref{eq-poscond}. When $C=0$, this reduces precisely to the case established by Miao and the author in \cite{MM17}.
\end{remark} 

\begin{remark}\label{rem-introminoo}
	There is no positive mass theorem or Penrose-like inequality in this case, as the counterexamples of Brendle, Marques and Neves, to Min-Oo's Conjecture \cite{BMN} demonstrate that a positive mass theorem in the usual sense does not hold here. This can be seen explicitly in Proposition \ref{prop-negm-minoo}, which demonstrates that there exist complete fill-ins of Bartnik data corresponding to the cosmological horizon in a negative mass Schwarzschild--de Sitter manifold.
\end{remark} 

As another application of the fill-in construction, Section \ref{ss-charge} establishes a charged Penrose-like inequality. That is, a lower bound on the total mass of an asymptotically flat manifold with boundary, equipped with an electric field, satisfying dominant energy conditions. This is stated precisely in Theorem \ref{thm-E-fill-in}.

This article is organised as follows. Section \ref{S-background} provides a brief background and overview of some related results before Section \ref{S-fillins} gives the general fill-in construction used here. Section \ref{S-fillinbounds} constructs the fill-ins required for positive and negative scalar curvature lower bounds, and finally Section \ref{S-apps} provides all of the applications to positive mass-type theorems.

\section{Brief Background}\label{S-background}

A Riemannian manifold with scalar curvature bounded below by a constant $C$ may be interpreted as time-symmetric initial data for the Einstein equations satisfying the dominant energy condition with cosmological constant $\Lambda=C/3$. In the case where $C\leq0$ the notion of total mass of an isolated system in the context of general relativity is well-understood, corresponding to the mass of an asymptotically hyperbolic manifold ($C<0$) or an asymptotically flat manifold ($C=0$). It is therefore no surprise that notions of mass in general relativity are connected with the problem of fill-ins with scalar curvature lower bounds.

More specifically, for $C=0$ the positive mass theorem \cite{Schoen-Yau79,Witten81} is a well-known foundational result in mathematical relativity, and for $C<0$ analogous positive mass theorems have been established for asymptotically hyperbolic manifolds \cite{CH-AHmass,WangMass}. That is, a complete asymptotically flat (resp. asymptotically hyperbolic) Riemannian manifold with scalar curvature bounded below by $0$ (resp. $C$, where $C<0$) has non-negative mass. The case where $C>0$, corresponding to a positive cosmological constant, does not have a standard notion of mass nor appear to have any prospect for a positive mass theorem to hold in general (see Remark \ref{rem-introminoo} and Proposition \ref{prop-negm-minoo}).

From the known positive mass theorems, it follows that if Bartnik data $(\S,g,H)$ can be realised as the boundary of an asymptotically flat or asymptotically hyperbolic manifold with scalar curvature lower bound of $C$ and negative mass (defined appropriately for the value of $C$), then it cannot admit a fill-in with scalar curvature bounded below by $C$. This is due to that fact that if such a fill-in exists then through a gluing procedure one could obtain an asymptotically flat or asymptotically hyperbolic manifold with minimimal surface boundary (if the boundary is non-empty) and negative mass, which would contradict the relevant positive mass theorem. Similarly, quasi-local positive mass theorems can also provide obstructions to the existence of fill-ins. For example, Shi and Tam's proof of the positivity of the Brown--York mass can be rephrased as a result on fill-ins as follows.

\begin{thm}[Shi--Tam \cite{ShiTam02}]\label{thm-shi-tam-1}
Let $(\S,g,H)$ be $(n-1)$-dimensional Bartnik data where $3\leq n\leq 7$ and with $H>0$, and such that $g$ is isometric to strictly convex closed hypersurface in $\R^n$. Then if
\be \label{eq-BY1}
	\int_\S Hd\mu_g > \int_\S H_o d\mu_g,
\ee 
where $H_o$ is the mean curvature of $(\S,g)$ isometrically embedded $\R^n$, there exists no fill-in of $(\S,g,H)$ with non-negative scalar curvature.
\end{thm} 

\begin{remark}
	When $n=3$, the existence of the required isometric embedding is well-known to be equivalent to the condition that $g$ has positive Gauss curvature \cite{Nirenberg, Pogorelov}. Furthermore, the dimensional restriction here is simply required to apply the positive mass theorem.
\end{remark}
Shi and Tam also proved a positivity statement for an asymptotically hyperbolic quasi-local mass in dimension $3$, which similarly gives the non-existence of fill-ins with negative scalar curvature lower bounds as follows.

\begin{thm}[Shi--Tam \cite{ShiTam07}, see also Wang--Yau \cite{WY}]\label{thm-shi-tam2}
	Let $(\S,g,H)$ be $2$-dimensional topologically spherical Bartnik data with $H>0$ and Gauss curvature $K_g> \frac{C}{6}$ for some $C<0$ that satisfies
	\bee 
	\int_\S (H_o-H)\cosh(-\sqrt{\frac{C}{6}}\,r)d\mu_g <0
	\eee 
	where $H_o$ is the mean curvature of $(\S,g)$ isometrically embedded in $3$-dimensional hyperbolic space with constant scalar curvature equal to $C$, and $r$ is the geodesic distance function from a fixed point in $\S$. Then $(\S,g,H)$ admits no fill-in with scalar curvature bounded below by $C$.
\end{thm} 
\begin{remark}
	The isometric embedding into hyperbolic space required for Theorem \ref{thm-shi-tam2} not only exists, but is unique up to an isometry of hyperbolic space \cite{dCW,Pogorelov}. It seems likely that a version of Theorem \ref{thm-shi-tam2} would also hold in higher dimensions, however some care should be taken in checking the details with particular attention given to the existence of required isometric embedding.
\end{remark}

It will be a recurring theme that the size of $H$ governs whether or not a fill-in exists. Jauregui \cite{Jeff11} shows this in a clear way with the following theorem.

\begin{thm}[Jauregui \cite{Jeff11}]\label{thm-Jauregui1}
	Let $(\S,g,H)$ be $2$-dimensional Bartnik data data with positive Gauss curvature and $H>0$. Then there exists $\lambda_o>0$ such that $(\S,g,\lambda H)$ admits a fill-in with non-negative scalar curvature for all $\lambda<\lambda_o$ and there exists no such fill-in for $\lambda>\lambda_o$.
\end{thm}

\begin{remark}
	Theorem \ref{thm-Jauregui1} also likely can be extended in a straightforward manner to higher dimension than $3$, although it relies on Theorem \ref{thm-shi-tam-1} so some extra hypotheses regarding the isometric embedding are likely required.
\end{remark}

In a similar spirit to Jauregui's result \cite{Jeff11}, Shi, Wang, Wei and Zhu \cite{SWWZ} establish the following result.
\begin{thm}[Shi--Wang--Wei--Zhu \cite{SWWZ}]\label{thm-swwz}
	Let $g$ be a metric on the sphere $\S=\mathbb{S}^{n-1}$ ($3\leq n\leq 7$) such that there exists a continuous path in the space of smooth positive scalar curvature metrics on $\mathbb{S}^{n-1}$ from $g$ to the standard round metric. Then there exists a constant $h_o$ depending on $g$ such that $(\S,g,H)$ admits no fill-in with non-negative scalar curvature for all $H>0$ satisfying
	\be \label{eq-nofillineq}
		\int_\S H d\mu_g > h_o.
	\ee 
\end{thm}

While there are several results demonstrating that if $H$ is too large in some sense then no fill-in can exist, it is difficult to explicitly quantify how large $H$ may be. It would be interesting to obtain an explicit, computable lower bound on $H$ in terms of $g$ for an obstruction of the existence of a fill-in, such as an explicit value for $h_o$ in \eqref{eq-nofillineq}.

\section{General Fill-In Construction}\label{S-fillins}
We construct fill-ins of Bartnik data $(\S,g,H)$ by constructing a metric $\gamma$ on the cylinder $\S\times I$ for some interval $I$, such that one boundary component induces the Bartnik data while the other is a minimal surface. This construction originates with \cite{MS} and has been used successfully in several related problems (see \cite{CCSurvey} for a survey).

We now give the general construction of these collars so we may refer to it later. Consider a metric of the form
\be \label{eq-metricform}
\gamma=A(x)^2dt^2+E(t)^2g
\ee 
where $A$ is a positive function on $\S$, $E$ is a positive function of $t$, and $g$ is a fixed metric on $\S$.

Computing the scalar curvature of the metric $\gamma$ we obtain
\begin{align}\begin{split} \label{eq-mainscalar}
E(t)^2R_\gamma=&\,R_g-2A^{-1}\Delta_\S A\\
&-(n-1)A^{-2}\( (n-2)E'(t)^2+2E(t)E''(t). \)\end{split}
\end{align}

We will be interested in choices of $E$ and $A$ that allow us to prescribe the mean curvature of some constant $t$ boundary surface, and ensure a prescribed lower bound on $R_\gamma$. The mean curvature of each slice with constant $t$ can computed directly as
\be \label{eq-Hformula}
H_t(x)=\frac{(n-1)E'(t)}{E(t)A(x)}.
\ee 
For convenience, we will choose the parameter $t$ such that $t=0$ is the surface with prescribed mean curvature, and we use $t<0$ for the fill-in, so that $t$ is increasing towards the outer boundary. We also must prescribe the induced metric on the outer boundary surface. To this end, we will always scale $E$ such that $E(0)=1$. Letting $H=H(x)$ be the mean curvature we wish to prescribe, we find
\be \label{eq-Aformula}
A(x)=\frac{(n-1)E'(0)}{H(x)}
\ee
and then \eqref{eq-mainscalar} becomes
\begin{align}\begin{split} \label{eq-mainscalar2}
E(t)^2R_\gamma=&\,R_g-2H\Delta_\S (H^{-1})\\
&-\frac{H^2}{(n-1)(E'(0))^2}\((n-2)E'(t)^2+2E(t)E''(t). \)\end{split}
\end{align}

In what follows, we will use this construction with different choices of $E$, coming from profile functions of model spherically symmetric metrics. We now demonstrate that the idea used to prove the main result of \cite{MM17} can be used to construct fill-ins with general scalar curvature lower bounds. Let $\(\S,g,H\)$ be Bartnik data where $\S$ is an $(n-1)$-sphere and $H$ is a positive function. We consider a metric of the form
\be \label{eq-metform}
\gamma=A(x)^2dt^2+\frac{u_m(t)^2}{r_o^2}g
\ee 
on $\S\times[t_o,0]$ where $r_o$ is the area radius of $g$ and $t_o<0$ will be determined later such that $\S\times\{t_o\}$ corresponds to a minimal surface boundary. Note that this is simply \eqref{eq-metricform} with $E(t)=\frac{u_m(t)}{r_o}$ and $g(t)=g$ a constant path. In \cite{MM17} the function $u_m$ was chosen to be a Schwarzschild profile function, however here we would like to include Schwarzschild--de Sitter and Schwarzschild--anti-de Sitter profiles too. Specifically $u_m$ is taken to be such that $u_m(0)=r_o$ and satisfies
\bee 
	u_m'(t)=\sqrt{1+\epsilon u(t)^2-\frac{2m}{u(t)^{n-2}}}
\eee
where $\epsilon\in\R$ depends on the desired scalar curvature lower bound and $m>0$ is some parameter. Specifically, for some $C\in\R$, we seek fill-ins with scalar curvature bounded below by $C$, and to do so we will choose $\epsilon=-\frac{C}{n(n-1)}$. From \eqref{eq-Hformula}, the mean curvature of each constant $t$ slice is given by
\be \label{eq-Hprescribe}
	H_t(x)=\frac{n-1}{A(x)u_m(t)}\sqrt{1+\epsilon u_m(t)^2-\frac{2m}{u_m(t)^{n-2}}}.
\ee 
It is important to note that for $\epsilon\geq0$, or $\epsilon<0$ and $m$ not too large, one may solve find a value of $u_m=r_H$ such that $1+\epsilon u_m(t)^2-\frac{2m}{u_m(t)^{n-2}}=0$, corresponding to an apparent horizon in the model manifold. In fact, if $\epsilon<0$ there are two such values of $u_m$ where this quantity vanishes, in which case $r_H$ is taken to be the smaller of the two, as the larger radius corresponds to a cosmological horizon in the Schwarzschild--de Sitter manifold. In particular, if $r_o>r_H$ then there is always a $t_o<0$ such that $H_{t_o}=0$. We will always ensure this is true so that the interior boundary of our collar is a minimal surface. 

Similarly, equation \eqref{eq-mainscalar2} for the scalar curvature gives

\begin{align} \begin{split}
	R_\gamma-C=\frac{r_o^2}{u_m^2}&\( R_g-2H\Delta H^{-1}-\frac{n-2}{n-1}H^2(1+\epsilon r_o^2-\frac{2m}{r_o^{n-2}})^{-1}\right.\\&\left.-  \(\frac{C}{r_o^2}+\frac{n}{n-1} H^2\epsilon (1+\epsilon r_o^2-\frac{2m}{r_o^{n-2}})^{-1} \)u_m^2 \).\label{eq-mainscalareq}\end{split}
\end{align}

When $C=0$ this is exactly what was considered in \cite{MM17}. We consider the cases of $C<0$ and $C\geq0$ separately, however the idea is the same in both cases. We choose $m>0$ such that the manifold $\( \S\times[t_o,0],\gamma \)$ has a minimal surface at the surface $t=t_o$, scalar curvature bounded below by $C$ and mean curvature of the surface $t=0$ prescribed as above.

\section{Fill-Ins With Scalar Curvature Lower Bounds} \label{S-fillinbounds}

\subsection{Negative scalar curvature lower bound}\label{SSec-negscalarbound}
We consider the case where the scalar curvature lower bound $C$ is negative, and first establish Theorem \ref{thm-neg}. Although Theorem \ref{thm-PMTAHB} is essentially a corollary of this, we reserve the proof of that until Section \ref{S-apps} to discuss with other applications of the fill-ins satisfying scalar curvature bounds.

\begin{proof}[Proof of Theorem \ref{thm-neg}]
	Consider the fill-in constructed above in Section \ref{S-fillins}, whose metric is given by \eqref{eq-metform}. In this case, we choose $\epsilon=-\frac{C}{n(n-1)}>0$ and the function $u_m$ is the radial profile function for an Schwarzschild--anti-de Sitter manifold with scalar curvature equal to $C$. By \eqref{eq-Hprescribe}, the fill-in constructed above has a minimal surface at some $t=t_o$ provided $m>0$. So we simply seek to choose $m>0$ such that the right-hand side of \eqref{eq-mainscalareq} is non-negative. A straightforward albeit non-optimal way to ensure that, is to impose
	\be \label{eq-negC1}
	R_g-2H\Delta H^{-1}-\frac{n-2}{n-1}H^2(1+\epsilon r_o^2-\frac{2m}{r_o^{n-2}})^{-1}\geq0
	\ee 
	and
	\be \label{eq-negC2}
	\frac{C}{r_o^2}+\frac{n}{n-1} H^2\epsilon (1+\epsilon r_o^2-\frac{2m}{r_o^{n-2}})^{-1} \leq 0.
	\ee 
	With our choice of $\epsilon$, \eqref{eq-negC2} becomes 
	\be
	m\leq \frac{r_o^{n-2}}{2}\(1+\epsilon r_o^2-\frac{H^2r_o^2}{(n-1)^2}  \),
	\ee 
	and \eqref{eq-negC1} can be expressed similarly as
	\be
	m\leq \frac{r_o^{n-2}}{2}\(1+\epsilon r_o^2-\frac{n-2}{n-1}\frac{H^2}{	R_g-2H\Delta H^{-1}}  \).
	\ee 
	In order for the fill-in to have a minimal surface inner boundary, rather than a cusp, we require $m>0$ so we obtain a fill-in provided
	\be \label{eq-cond-negC}
	\max\left\{\max_\S \frac{n-2}{n-1}\,\frac{H^2}{R_g-2H\Delta H^{-1}}, \max_\S \frac{H^2r_o^2}{(n-1)^2}  \right\}<1+\epsilon r_o^2.
	\ee 
	
\end{proof}
Note that in the case where $C=0$, \eqref{eq-negC2} is trivially satisfied and this reduces to what was shown in \cite{MM17}.

\subsection{Positive scalar curvature lower bound}
We next turn to consider a positive lower bound on the scalar curvature and prove Theorem \ref{thm-pos}.

\begin{proof}[Proof of Theorem \ref{thm-pos}]
	We again use the same fill-in metric \eqref{eq-metform} from Section \ref{S-fillins}, however in this case with $C>0$ and $\epsilon=-\frac{C}{n(n-1)}<0$.
	In this case, the model space is the Schwarzschild--de Sitter family of manifolds. In this case, a minimal surface is again located where $u'_m=0$, however here we must take a little more care with the roots of
	\be \label{eq-desitroots}
	1+\epsilon x^2-\frac{2m}{x^{n-2}}.
	\ee
	
	When $\epsilon\geq0$ we find that there is only one root, and therefore one minimal surface in the model space, which represents a black hole's horizon. However, when $\epsilon<0$ and provided that $m$ is not too large, \eqref{eq-desitroots} has two real positive roots, $0<r_+<r_-$. These correspond to a black hole horizon at $r_+$ and a cosmological horizon at $r_-$ in the model Schwarzschild--de Sitter manifold. This model is a compact manifold with two connected minimal surface boundary components, one sphere at $r_+$ and another sphere at $r_-$. We will choose $m$ arbitrarily small but positive and require that the area radius of $g$, $r_o$ satisfy $r\in(r_+,r_-)$. Since the roots $r_+$ and $r_-$ tend to $0$ and $ \frac{1}{\sqrt{-\epsilon}} $ respectively, as $m\to 0$, the condition $r_o^2<\frac{n(n-1)}{C}$ guarantees $r_+>2m>0$ for sufficiently small $m$. In particular, we have $r_o\geq u_m(t)\geq r_+>2m>0$. Turning back to the scalar curvature equation, \eqref{eq-mainscalareq} implies that $R_\gamma\geq C$ is equivalent to

	\bee 
	R_g-2H\Delta H^{-1}-\frac{H^2}{n-1}(1+\epsilon r_o^2-\frac{2m}{r_o^{n-2}})^{-1}\(n-2-n\epsilon u_m^2\)+\frac{n(n-1)\epsilon u_m^2}{r_o^2}\geq 0.
	\eee
	
	As we do not expect to obtain anything optimal by this method, it will suffice to make a crude estimate using the fact that $u_m$ is positive. Specifically, we ask that
	
	\bee 
	R_g-2H\Delta H^{-1}-\frac{n-2}{n-1}H^2(1+\epsilon r_o^2-\frac{2m}{r_o^{n-2}})^{-1}\geq 0,
	\eee 
	which is ensured by choosing
	\bee 
	0<m\leq \frac{r^{n-2}}{2}\( 1+\epsilon r_o^2-\frac{n-2}{n-1}\frac{H^2}{R_g-2H\Delta H^{-1}}\).
	\eee 
	
	That is, there exists a fill-in with minimal surface boundary and scalar curvature bounded below by $C$, provided that
	\be \label{eq-pos-loc-pmt-cond}
	\frac{n-1}{n-2}\frac{R_g-2H\Delta H^{-1}}{H^2}<	 \(1-\frac{C}{n(n-1)} r_o^2\)^{-1}.
	\ee
\end{proof}

	Note that equation \eqref{eq-pos-loc-pmt-cond} amounts to a strong, point-wise quasi-local positive mass assumption, which is somewhat natural in the context of non-positive scalar curvature lower bounds given the positive mass theorem. However, as mentioned above, the analogous positive mass theorem does not hold for positive scalar curvature lower bounds. That is, fill-ins with positive scalar curvature bounds should exist under far weaker assumptions than in the case of non-positive scalar curvature bounds. 
	
	One can see this from the counterexample to Min-Oo's conjecture constructed by Brendle, Marques and Neves in \cite{BMN}. Therein they establish the existence of compact manifolds with scalar curvature bounded below by $n(n-1)$, with a neighbourhood of the boundary isometric to a neighbourhood of the boundary of the $n$-dimensional hemisphere, and strictly positive scalar curvature somewhere on the interior. From a perturbation of this counterexample one can obtain a manifold with the same scalar curvature lower bound with a neighbourhood of the boundary isometric to a neighbourhood of the cosmological horizon (unstable minimal surface) in a negative mass Schwarzschild--de Sitter manifold. Although this is a fairly obvious consequence of the main results of \cite{BMN}, as it does not appear to be recorded explicitly, we state it here for completeness.
\begin{prop}\label{prop-negm-minoo}
	There exists a compact Riemannian manifold $(M,g)$ with boundary, having scalar curvature $R_g\geq 6$, with the inequality strict ($R_g>6$) on an open subset, such that there exists a subset $\Omega\supset\p M$ isometric to a neighbourhood of the boundary of a Schwarzschild--de Sitter manifold with negative mass.
\end{prop}
\begin{proof}
	By Theorem 4 of \cite{BMN} there exists a metric $g_o$ on the hemisphere $S_+^{n}$ with scalar curvature $R_{g_o}>n(n-1)$ everywhere, $g_o$ is exactly equal to the standard round metric on $\p S_+^n=S^{n-1}$, and the outward-pointing mean curvature $H$ of $\p S_+^n$ with respect to $g_o$ is strictly negative. By a small rescaling, one can also obtain a metric $g_\varepsilon$ that also satisfies $R_{g_\varepsilon}>n(n-1)$ with negative mean curvature on the boundary, such that $g_\varepsilon$ restricted to $\p S_+^n$ is round with area $A_\varepsilon>4\pi$. Note that the cosmological horizon boundary of a Schwarzschild–anti-de Sitter manifold of mass $m$ has area $\widetilde A_m$ satisfying
	\bee 
	m=\frac12\( \frac{\widetilde A_m}{\omega_{n-1}} \)\(1-\frac{\widetilde A_m}{\omega_{n-1}} \).
	\eee 
	That is, the metric $g_\varepsilon$ restricted to the boundary $\p S_+^n$ is equal to the boundary metric for a negative mass Schwarzschild–anti-de Sitter manifold. One can therefore apply Theorem 5 of \cite{BMN} to obtain the result. Note that although the negative mass Schwarzschild–anti-de Sitter manifolds are singular at a point, this does not affect the application of Theorem 5 of \cite{BMN} since it is purely a local construction.
\end{proof}

\begin{remark}\label{rem-conj-minoo2}
	The above is simply a perturbative construction, so the manifolds obtained correspond to a (negative) mass parameter very close to zero. It would be an interesting question to ask for Bartnik data $(S^{n-1},g_{0},H=0)$, with $g_{o}$ round, how large can $|S^{n-1}|_{g_o}$ be and still admit a fill-in with scalar curvature bounded below by $n(n-1)$.
\end{remark}

\section{Applications}\label{S-apps}

\subsection{The asymptotically hyperbolic positive mass theorem with boundary}\label{subsec-AHPMT}
As mentioned above, in \cite{MM17}, the fill-ins constructed were used to prove a ``Penrose-like" inequality. In particular, it was shown that for an asymptotically flat manifold with boundary $\S$ satisfying the hypotheses of Theorem \ref{thm-pos} with $C=0$, there exists a fill-in with minimal surface boundary whose area $A$ satisfies
\be \label{eq-MM17}
	\(\frac{A}{\omega_{n-1}}\)^{\frac{n-2}{n-1}}=\( \frac{|\S|}{\omega_{n-1}}\)^{\frac{n-2}{n-1}}\( 1-\frac{n-2}{n-1}\frac{H^2}{R_g-2H\Delta H^{-1}} \),
\ee 
and then via the Riemannian Penrose inequality we obtain that the ADM mass is bounded below by the right-hand side of \eqref{eq-MM17}.

An analogous result follows by the same reasoning for asymptotically hyperbolic manifolds with boundary using the fill-in constructed in Section \ref{SSec-negscalarbound} if one assumes the Riemannian Penrose inequality holds in this case. While this version of the Riemannian Penrose inequality is only conjectured, and not yet established, it is known that the positive mass theorem holds for asymptotically hyperbolic manifolds with minimal surface boundary \cite{CH-AHmass}. So we obtain the following positive mass theorem for asymptotically hyperbolic manifolds with boundary.

\begin{proof}[Proof of Theorem \ref{thm-PMTAHB}]
	Let $(M,\gamma)$ satisfy the hypotheses of Theorem \ref{thm-PMTAHB} with boundary Bartnik data $(\S,g,H)$, and let $\gamma_\Omega$ be the fill-in metric on $\Omega=\S\times[t_o,1]$ constructed for the proof of Theorem \ref{thm-neg}, filling in this Bartnik data. We can attach $(\Omega, \gamma_\Omega)$ to $(M,\gamma)$ along their matching Bartnik data to form a manifold with corner \`a la Miao \cite{Miao02} and then double it along the minimal surface to form a complete asymptotically hyperbolic manifold with corners and two ends. The conclusion follows now from the positive mass theorem with corners, which has been established in the asymptotically hyperbolic case by Bonini and Qing \cite{BoniniQing}.
\end{proof}

\begin{remark}
If the conjectured Riemannian Penrose inequality for asymptotically hyperbolic manifolds were established, then one could use a suitable version of it for manifolds with corners to obtain an improved lower bound on the mass of an asymptotically hyperbolic manifold with boundary. In particular, we have the following.

Let $(M,\gamma)$ be an asymptotically hyperbolic $n$-manifold with scalar curvature bounded below by $C=-\epsilon n(n-1)$ and interior boundary $\S$, and define the quantity
\be \label{eq-chidefn} 
\chi=\frac{1}{n-1}	\max\left\{\max_\S \frac{(n-2)H^2}{R_g-2H\Delta H^{-1}}, \max_\S \frac{H^2r_o^2}{(n-1)}  \right\}.
\ee
If $\chi<1+\epsilon r_o^2$, where $r_o$ is the area radius of $\S$, then assuming the asymptoticaly hyperbolic Penrose inequality holds (on a manifold with corners), we would conclude that the total mass of $(M,\gamma)$ is bounded below by $\frac12r_o^{n-2}\( 1+\epsilon r_o^2-\chi \)$. 
\end{remark}

	We state this as a remark and omit a formal proof of this statement, as we require a precise statement of the appropriate Riemannian Penrose inequality, which remains an open problem. However, a sketch is provided as we will refer back to at the end of Section \ref{ssec-entropy}.
\begin{proof}[Sketch of proof]	
	Choose $m= \frac{r_o^{n-2}}{2}\( 1+\epsilon r_o^2-\chi\)>0$ and obtain a fill-in with minimal surface boundary as in Section \ref{SSec-negscalarbound}. One can quickly check from the form of the metric \eqref{eq-metform} that the area radius $r_H$ of this minimal surface satisfies
	\bee 
	1+\epsilon r_H^2-\frac{2m}{r_H^{n-2}}=0.
	\eee 
	From our choice of $m$, we therefore obtain the relationship
	\be\label{eq-minsurfaceneg} 
	\frac12r_H^{n-2}\( 1+\epsilon r_H^2 \)=\frac12r_o^{n-2}\( 1+\epsilon r_o^2-\chi \),
	\ee 
	where the left-hand side of the equation is exactly the lower bound on the total mass of an asymptotically hyperbolic manifold conjectured by the Riemannian Penrose inequality. This fill-in can then be glued to $(M,\gamma)$ and the conclusion would follow from the Riemannian Penrose inequality.
\end{proof}
\begin{remark}
	There are two different conjectured versions of the asymptotically hyperbolic Penrose inequality corresponding roughly to whether the asymptotically hyperbolic manifold is being viewed as time-symmetric initial data for a spacetime with negative cosmological constant, or as an asymptotically hyperbolic slide of an asymptotically flat spacetime. The former considers the horizon to be a minimal surface, whereas the latter considers a surface of constant mean curvature equal to $n-1$. Here we consider the former version, and while it seems that many suspect it to hold there is some reason to suspect that perhaps only the latter is true in general. We do not wish to speculate on the conjecture here, however it is worthwhile noting that the construction given above could in principle be used to construct a counterexample if one exists. That is, if one can find an asymptotically hyperbolic manifold (perhaps only defined near infinity) containing a surface on which the right-hand side of \ref{eq-minsurfaceneg} is larger than the total mass, then after gluing, a counterexample to the asymptotically hyperbolic Penrose inequality would be constructed. However, the author has attempted this with no success.
\end{remark}

\subsection{A Penrose-like inequality with electric charge}\label{ss-charge}

In the framework of general relativity, it is common to consider the Einstein equations coupled to other equations governing the matter content of the universe to model. Considering $(M,\gamma)$ as initial data from the perspective of general relativity, we can add an electric field $E$ -- a vector field on $M$ -- to describe initial data for the coupled Einstein--Maxwell system, gravity coupled to the electric field. In this section we will only consider the case of vanishing cosmological constant, which in the preceding sections equated to manifolds with non-negative scalar curvature. This restriction to only considering vanishing cosmological constant is not required to construct the fill-ins, however the estimates become much messier and these ``charged'' fill-ins are not of particular interest independent of a Penrose-like inequality. However, such an inequality cannot hold in the positive cosmological constant case and remains out of reach for the negative cosmological constant case, for the reasons discussed above. On the other hand, the scalar curvature lower bound considered here is not zero and in fact depends on the electric field $E$. Specifically, we will require
\be \label{eq-QDEC} 
R_g\geq (n-1)(n-2)|E|_g^2.
\ee 

The divergence of the electric field corresponds to the charge of any matter source terms, and in order to later apply a charged Riemannian Penrose inequality \cite{Jang79,KWY-exts-2015,KWY-2017,MeChargedPenrose} we will ask that this vanishes. It is not strictly required that the $\nabla\cdot E=0$ for the charged Riemannian Penrose inequality to hold (see \cite{MeChargedPenrose}), however it will nevertheless be fruitful to impose this for the fill-ins constructed. In the preceding sections, a fill-in is taken as a manifold with boundary metric and mean curvature prescribed since this is the appropriate boundary condition to glue the fill-in to an exterior manifold while preserving the scalar curvature condition. However, when an electric field is present we would also like to preserve the sign of $\text{div}(E)$ in a distributional sense when performing the gluing, which amounts to matching $\phi=E\cdot n$ the normal component of the electric field on $\S$. Therefore the appropriate Bartnik data for including electric charge is the triple $(\S,g,H,\phi)$ (see, for example, Section 3 of \cite{CM}).

We again consider a metric of the form \eqref{eq-metform} except now use a Reissner--Nordstr\"om manifold as our model and profile curve
\be
\gamma=A(x)^2dt^2+\frac{v_{m,Q}(t)^2}{r_o^{2}}g,
\ee 
where $v_{m,Q}$ satisfies $v_{m,Q}(0)=r_o$ and
\bee 
	v'_{m,Q}(t)=\sqrt{1+\frac{Q^2}{v(t)^{2(n-2)}}-\frac{2m}{v(t)^{n-2}}}.
\eee 
In what follows we will use the shorthand $v=v_{m,Q}$ for the sake of presentation. We will also assume $H$ is constant here, as it simplifies the computations considerably and does not change the qualitative properties of the estimate we obtain. We again choose $A(x)$ according to \eqref{eq-Aformula} and compute the scalar curvature of $\gamma$ similarly to \eqref{eq-mainscalareq} to obtain
\be
	R_\gamma=\frac{r_o^2}{v^2}\( R_g-\frac{n-2}{n-1}H^2\(1+ \frac{Q^2}{r_o^{2(n-2)}}-\frac{2m}{r_o^{n-2}}\)^{-1}\(1-\frac{Q^2}{v^{2(n-2)}}\)\).
\ee
We then set the electric field as
\bee 
	E=\frac{r_o^{n-1} \phi}{v^{n-1} A}\p_t,
\eee 
which is easily checked to be divergence-free. Then we see that the appropriate energy condition, $R_g\geq (n-1)(n-2)|E|_g^2$, is equivalent to
\bee
R_g-\frac{n-2}{n-1}H^2\(1+ \frac{Q^2}{r_o^{2(n-2)}}-\frac{2m}{r_o^{n-2}}\)^{-1}\(1-\frac{Q^2}{v^{2(n-2)}}\)-\frac{(n-1)(n-2)r_o^{2(n-2)}\phi^2}{v^{2(n-2)}}\geq0.
\eee
Assuming $H>(n-1)\phi>0$, we may choose $Q$ in such a way to cancel out the  $v^2$ terms. Namely, we set
\begin{align*} 
	Q^2&=\frac{(n-1)^2r_o^{2(n-2)}\hat\phi^2}{H^2}\(1+ \frac{Q^2}{r_o^{2(n-2)}}-\frac{2m}{r_o^{n-2}}\)\\
	Q^2&=(n-1)^2r_o^{2(n-2)}\hat\phi^2\( 1-\frac{2m}{r_o^{2(n-2)}} \)\(H^2-(n-1)^2\hat\phi^2\)^{-1},
\end{align*}
where we write $\hat\phi=\max_\Sigma(\phi)$ for the sake of presentation. This leaves $m<\frac{r_o^{2(n-2)}}{2}$ as a free parameter, which we will choose to ensure
\bee 
	R_g-\frac{n-2}{n-1}H^2(1+ \frac{Q^2}{r_o^{2(n-2)}}-\frac{2m}{r_o^{n-2}})^{-1}\geq0.
\eee  
With our choice of $Q$, this is equivalent to
\bee 
R_g-\frac{n-2}{n-1}\,\frac{\(H^2-(n-1)^2\hat\phi^2\)}{\( 1-\frac{2m}{r_o^{n-2}} \)}\geq0,
\eee  
so we choose
\bee
	m=\frac{r_o^{n-2}}{2}\(1- \frac{n-2}{n-1}\,\frac{H^2-(n-1)^2\hat\phi^2}{\min_\Sigma(R_g)} \).
\eee
Our fill-in now satisfies the appropriate energy condition for the charged Riemannian Penrose inequality \eqref{eq-QDEC}, however we have yet to check that the metric is non-singular. In fact, we require that $m>Q$ to ensure that $v'$ vanishes somewhere and therefore providing us with a minimal surface boundary. Note that with our choice of $m$, we have

\begin{align*} 
	Q^2&=r_o^{2(n-2)}\( \frac{n-2}{n-1}\, \frac{H^2-(n-1)^2\hat\phi^2}{\min_\Sigma(R_g)} \)\(\frac{H^2}{(n-1)^2\hat\phi^2}-1\)^{-1}\\
		Q^2&=\frac{(n-1)(n-2)\hat\phi^2r_o^{2(n-2)}}{\min_\Sigma(R_g)}.
\end{align*}
In order to ensure $m>Q$, we calculate the difference $m-Q$ from the above expressions to obtain
\begin{align} \begin{split}\label{eq-q-m}
	\frac{2R_g}{r_o^{(n-2)}}(m-Q)=&\,\min_\Sigma(R_g)-\frac{n-2}{n-1}H^2+(n-2)(n-1)\hat\phi^2\\&-2\hat\phi\sqrt{(n-1)(n-2)\min_\Sigma(R_g)}.\end{split}
\end{align}
The right-hand side of \eqref{eq-q-m} is a quadratic expression in $\sqrt{\min_\Sigma(R_g)}$, so we can directly check that this is positive when
\bee 
	\min_\Sigma(R_g)>\frac{n-2}{n-1}\( (n-1)\hat\phi+H \)^2.
\eee 

Note that when $\phi\equiv0$ this reduces to the condition for the existence of fill-in for the uncharged case \cite{MM17}. By construction, we have proven the following.
\begin{thm}\label{thm-E-fill-in}
	Let $(\S,g,H,\phi)$ be charged Bartnik data with constant $H$ satisfying $H>(n-1)\phi>0$ and satisfying
	\bee 
		\min_\Sigma(R_g)>\frac{n-2}{n-1}\( (n-1)\max_\Sigma(\phi)+H \)^2
	\eee 
	then there exists a metric $\gamma$ and divergence-free vector field $E$ on $M=\Sigma\times[0,1]$ satisfying $R_g\geq (n-1)(n-2)|E|_\gamma^2$, such that one boundary component is a minimal surface and on the other the induced metric is $g$, outward-pointing mean curvature is $H$, and the outward normal component of $E$ is $\phi$.
\end{thm}

From this fill-in, we are able to prove an electrically charged Penrose-like inequality (and charged positive mass theorem for manifold with boundary).
\begin{thm}\label{thm-chargedpmt}
	Let $(M,\gamma,E)$ be an asymptotically flat manifold with charge of dimension $3\leq n\leq7$, and boundary $\S$ with charged Bartnik data $(\S,g,H,\phi)$, satisfying $\nabla\cdot E\geq0$ and
	
		\bee 
	\min\limits_\S R_g>\frac{n-2}{n-1}\( (n-1)\max\limits_\S\phi+\max\limits_\S H \)^2.
	\eee 
	Then
	\bee 
		\m_{ADM}\geq m+\frac{Q_\Sigma^2-Q^2}{m+\sqrt{m^2-Q^2}},
	\eee 
 	where $Q_\Sigma=\frac{1}{\omega_{n-1}}\int_S\phi\,d\mu_g$ is the electric charge on $\S$ in $M$, and the parameters $m$ and $Q$ are given by
 	\begin{align}\begin{split}\label{eq-mQdefn}
 		m&=\frac{r_o^{n-2}}{2}\(1- \frac{n-2}{n-1}\,\frac{H^2-(n-1)^2\max_\Sigma\phi^2}{\min_\Sigma(R_g)} \)\\
 		Q^2&=\frac{(n-1)(n-2)\max_\Sigma\phi^2r_o^{2(n-2)}}{\min_\Sigma(R_g)}.\end{split}
 	\end{align}
 	Furthermore, if $\nabla\cdot E\equiv 0$ then we have
 	\be \label{eq-loc-QM}
 	\m_{ADM}\geq |Q_\Sigma|.
 	\ee
\end{thm}
\begin{remark}
	The expression for $m$ given by \eqref{eq-mQdefn} can be compared to the charged Hawking mass, while the expression $Q_\Sigma^2-Q^2$ can be seen to vanish when $(S,g)$ is a round sphere and $\phi$ is constant.
\end{remark}
\begin{proof}
	We apply Theorem \ref{thm-E-fill-in} to construct a fill-in of the Bartnik data $(\S,g,H_o,\phi)$, where $H_o=\max_\Sigma(H)$. We then obtain a (charged) manifold with corner by attaching the fill-in to $M$ and can apply the charged Riemannian Penrose inequality for manifolds with corners established in \cite{CM}. Note that the results in \cite{CM} are stated only in dimension $3$ only, however it is clear from the proof that the charged Riemannian Penrose inequality with corners holds in dimension up to $7$ (see remark \ref{rem-CMdims} below). Specifically, from Theorem 1.3 of \cite{CM} we have
	\be \label{eq-firstE}
		\m_{ADM}(M,\gamma)\geq \frac{r_H^{n-2}}2 \left(1+\frac{Q_\Sigma^2}{r_H^{2(n-2)}}\right),
	\ee  
	where $r_H$ is the area radius of the minimal surface boundary.
	
	From the definition of the profile function $v$ used in the fill-in, we have that the minimal surface occurs when $v'=0$, which implies
	\bee 
	1+\frac{Q^2}{r_H^{2(n-2)}}-\frac{2m}{r_H^{n-2}}=0,
	\eee 
	and then from \eqref{eq-firstE}, we have
	\be \label{eq-secondE}
		\m_{ADM}(M,\gamma)\geq m+\frac{Q_\Sigma^2-Q^2}{r_H^{n-2}},
	\ee
	with 
	\bee 
		r_H^{n-2}=m+\sqrt{m^2-Q^2}.
	\eee 
	Finally note that \eqref{eq-loc-QM} follows from applying the charged positive mass theorem \cite{PosMassQBH, dLGLS} instead of the charged Riemannian Penrose inequality.
\end{proof}

\begin{remark}\label{rem-CMdims}
	The charged Riemannian Penrose inequality with corners established in \cite{CM} is presented in dimension 3, following \cite{Miao02}. However, as illustrated in the appendix of \cite{MM17}, the argument holds up to dimension $7$ (where the standard problem concerning the regularity of minimal surfaces prevents it immediately being generalised to higher dimensions than that). Naturally, there are dimensional constants in the inequality used, which we are careful to correctly include here.  
\end{remark}

\begin{remark}
	As with the other inequalities established by this method, it is expected that there is room slightly improve the inequality, however comparing it to the lower bound on ADM mass in terms of the charged Hawking mass in dimension 3, one sees that the method is unlikely to achieve an optimal inequality.
\end{remark}

\subsection{Engelhardt--Wall Outer Entropy and Bray's Inner Bartnik Mass} \label{ssec-entropy}
Recently, Wang \cite{WangEntropy} noted that the concept of outer entropy due to Engelhardt and Wall \cite{EW} in the context of the AdS/CFT correspondence is essentially the same concept as Bray's inner Bartnik mass \cite{BCinnermass}. The former was formulated from the perspective of the AdS/CFT correspondence for asymptotically hyperbolic manifolds while the latter was formulated from a purely geometric perspective for asymptotically flat manifolds. In particular, the outer entropy is equivalent to an asymptotically hyperbolic analogue of the inner mass, rather than the standard one. Nevertheless, at the heart of both is the problem of constructing fill-ins of Bartnik data with a minimal surface boundary, and taking the supremum of the minimal surface area over an appropriate class of fill-ins\footnote{Technically, the inner mass is defined using fill-ins that extend out to another asymptotic end, but this distinction is minor.}.

The Penrose-like inequality obtained in \cite{MM17} in fact was first motivated by considerations of the Bartnik--Bray inner mass. Since this mass is taken as a supremum, we immediately conclude that the inner Bartnik mass of Bartnik data $(\S,g,H)$ is bounded below by the right-hand side of \eqref{eq-MM17}.

While an asymptotically hyperbolic analogue of the standard Bartnik mass has been recently investigated \cite{CCM}, to the best of the author's knowledge an asymptotically hyperbolic inner Bartnik mass has not been considered in the literature. Nevertheless, there is an obvious analogue one could consider, which is the one equivalent to the outer entropy. Namely, we define the asymptotically hyperbolic inner Bartnik mass of given Bartnik data $(\S,g,H)$ as the supremum, taken over the set of all fill-ins with scalar curvature bounded below by $-\epsilon n(n-1)$ with no closed minimal surfaces except for a minimal surface boundary, of the quantity
\bee 
\frac12r^{n-2}\( 1+\epsilon r^2 \) 
\eee 
where here $r$ is the area radius of the minimal surface boundary. It follows from \eqref{eq-minsurfaceneg}  that this asymptotically hyperbolic inner Bartnik mass of some data $(\S,g,H)$ is bounded below by
\be \label{eq-forr}
	\frac12r_o^{n-2}\( 1+\epsilon r_o^2-\chi \)
\ee 
where $\chi$ is given by \eqref{eq-chidefn} and $r_o$ is the area radius of $g$. Then observation of Wang connects this to the Engelhardt--Wall outer entropy, which is simply the supremum of the area of the minimal surface over the same set of fill-ins. That is the lower bound for the outer entropy is $\omega_{n-1}r_H^{n-1}$ where $r_H$ satisfies \eqref{eq-minsurfaceneg}.

\bibliographystyle{abbrv}

\end{document}